\documentclass[a4paper,12pt]{article}

\usepackage[top=3.0cm,bottom=3.0cm,left=2.25cm,right=2.25cm]{geometry}
\usepackage{amsmath,amsthm,mathrsfs,graphicx,amsfonts}
\usepackage{bm}
\usepackage{cases}
\usepackage[english]{babel}
\usepackage{amsmath,amsthm}
\usepackage{amsfonts}
\usepackage{latexsym}
\usepackage{graphicx}
\usepackage{txfonts}
\usepackage[numbers,sort&compress]{natbib}
\usepackage[natural]{xcolor}
\usepackage{rotating}
\usepackage{mathtools}
\usepackage{enumitem}

\newtheorem{lem}{Lemma}[section]
\newtheorem{thm}[lem]{Theorem}
\newtheorem{cor}[lem]{Corollary}

\newtheorem{conj}[lem]{Conjecture}

\newcommand{\NP}{\mathbf{N}_{\mathcal P}}
\newcommand{\cC}{\mathcal C}
\newcommand{\supp}{\operatorname{supp}}
\newcommand{\defeq}{\vcentcolon=}

\usepackage{authblk}

\title{\bf A sharp asymptotic bound for odd cycles in planar graphs}
\author{
Zhen Liu\footnote{Email: 1552580575@qq.com},
~Chuanshu Wu\footnote{Email: cswu97@126.com (Corresponding author)}\\
{\small Center for Discrete Mathematics, Fuzhou University, Fujian, 350003, China}}

\date{\today}

\begin{document}

\maketitle

\begin{abstract}
For graphs $G$ and $H$, let $\mathbf N(G,H)$ denote the number of unlabeled, not necessarily induced copies of $H$ in $G$, and let $\NP(n,H)$ be the maximum of $\mathbf N(G,H)$ over all $n$-vertex planar graphs $G$. We prove that, for every fixed integer $m\geq 3$,
$$\NP(n,C_{2m+1})=2m\left(\frac{n}{m}\right)^m+O_m\!\left(n^{m-1/5}\right).$$
The proof uses a sharp weighted cycle--path inequality for edge probability measures on finite complete graphs. This strengthens a conjecture of Heath, Martin, and Wells and, together with their reduction lemma, yields the stated asymptotic formula.


\end{abstract}

\section{Introduction}

All graphs in this paper are finite and simple. Let $P_k$ and $C_k$ denote, respectively, the path and the cycle on $k$ vertices. For graphs $G$ and $H$, let $\mathbf N(G,H)$ denote the number of unlabeled, not necessarily induced copies of $H$ in $G$. For a planar graph
$H$, define
\[\NP(n,H)\defeq \max\{\mathbf N(G,H):G\text{ is an $n$-vertex planar graph}\}.\]

The problem of maximizing the number of copies of a fixed graph in a planar graph goes back to Hakimi and Schmeichel \cite{HakimiSchmeichel1979} and Alon and Caro \cite{AlonCaro1984}.
Hakimi and Schmeichel determined $\NP(n,C_3)$ and $\NP(n,C_4)$ exactly, while Huynh, Joret, and Wood \cite{HuynhJoretWood2022} determined the
order of magnitude of $\NP(n,H)$ for every fixed planar graph $H$.
A natural next problem is to determine the sharp leading coefficient and, when possible, the exact value.

The exact formula for $\NP(n,C_5)$ conjectured by Hakimi and Schmeichel \cite{HakimiSchmeichel1979} was later proved by Gy\H{o}ri, Paulos, Salia,
Tompkins, and Zamora \cite{GyoriEtAl2025}.
For longer even cycles, Cox and Martin determined $\NP(n,C_6)$, $\NP(n,C_8)$, $\NP(n,C_{10})$, and $\NP(n,C_{12})$ asymptotically \cite{CoxMartin2022,CoxMartin2023}.
They also introduced weighted reduction lemmas that reduce planar subgraph-counting problems to weighted optimization problems on complete
graphs, and conjectured the sharp leading term for every even cycle.
Lv, Gy\H{o}ri, He, Salia, Tompkins, and Zhu \cite{LvEtAl2024} proved this conjecture, showing that, for every fixed $m\geq 3$,
\[\NP(n,C_{2m})=\left(\frac{n}{m}\right)^m+o(n^m).\]

Heath, Martin, and Wells \cite{HeathMartinWells2025} developed an analogous weighted reduction for odd cycles. For every $m\geq 3$, a balanced
construction based on a copy of $C_m$ gives
\begin{equation}\label{eq:lower}
  \NP(n,C_{2m+1}) \geq 2m\left(\frac{n}{m}\right)^m-O_m(n^{m-1}).
\end{equation}
Moreover, for $m\in\{3,4\}$, they proved that
$\NP(n,C_{2m+1})=2m\left(\frac{n}{m}\right)^m+O_m\!\left(n^{m-1/5}\right),$
whereas for $m\geq 5$, they showed that
$\NP(n,C_{2m+1})\leq 2.7\,m\left(\frac{n}{m}\right)^m+O_m\!\left(n^{m-1/5}\right).$
Wang, Gy\H{o}ri, and He \cite{WangGyoriHe2026} recently determined, for every fixed $m\geq 3$ and all sufficiently large $n$, the exact maximum
number of copies of $C_{2m+1}$ in planar graphs containing no shorter odd cycle.

In this paper, we determine $\NP(n,C_{2m+1})$ asymptotically by proving a sharp weighted cycle--path inequality.
To formulate this inequality, we adopt the weighted notation of Heath, Martin, and Wells \cite{HeathMartinWells2025}.
Let $K$ be a finite complete graph. An \emph{edge probability measure} on $K$ is a function $\mu:E(K)\to[0,1]$ such that $\sum_{e\in E(K)}\mu(e)=1.$
Let $\Delta^K$ denote the set of all edge probability measures on $K$, and let $\supp\mu\defeq\{e\in E(K):\mu(e)>0\}.$
For a subgraph $H'\subseteq K$, define
\[\mu(H')\defeq\prod_{e\in E(H')}\mu(e).\]
For a graph $H$, let $\cC(K,H)$ denote the set of all unlabeled, not necessarily induced copies of $H$ in $K$, and define
\[\beta(\mu;H)\defeq \sum_{H'\in\cC(K,H)}\mu(H').\]

The following lemma reduces the planar problem to a weighted optimization problem.

\begin{lem}[Heath, Martin, and Wells~\cite{HeathMartinWells2025}]
\label{lem:reduction}
For every $n$-vertex planar graph $G$ and every integer $m\geq 3$, there
exist a finite complete graph $K$ and a measure $\mu\in\Delta^K$ such that
\[\mathbf N(G,C_{2m+1})\leq \bigl(2m\cdot\beta(\mu;C_m)+\beta(\mu;P_{m+1})\bigr)n^m +O_m\!\left(n^{m-1/5}\right).\]
\end{lem}

Heath, Martin, and Wells proposed the following conjecture.

\begin{conj}[Heath, Martin, and Wells~\cite{HeathMartinWells2025}]
\label{conj:HMW}
For every finite complete graph $K$, every integer $m\geq 3$, and every $\mu\in\Delta^K$,
\[2m\cdot\beta(\mu;C_m)+\beta(\mu;P_{m+1})\leq \frac{2}{m^{m-1}}.\]
Equality holds if and only if $\mu$ is the uniform measure on the edge set of a copy of $C_m$ in $K$.
\end{conj}

They further proposed the stronger version obtained by replacing the coefficient of $\beta(\mu;P_{m+1})$ by $2$. Our main result proves this stronger inequality.

\begin{thm}\label{thm:weighted}
For every finite complete graph $K$, every integer $m\geq 3$, and every $\mu\in\Delta^K$,
\begin{equation}\label{eq:strong}
  2m\cdot\beta(\mu;C_m)+2\cdot\beta(\mu;P_{m+1}) \leq \frac{2}{m^{m-1}}.
\end{equation}
Moreover, equality in the inequality of Conjecture~\ref{conj:HMW} holds if and only if $\mu$ is the uniform measure on the edge set of a copy of $C_m$ in $K$.
\end{thm}

Combining Theorem~\ref{thm:weighted} with Lemma~\ref{lem:reduction} and the lower bound \eqref{eq:lower}, we obtain the following corollary.

\begin{cor}\label{cor:planar}
For every fixed integer $m\geq 3$,
\[\NP(n,C_{2m+1})=2m\left(\frac{n}{m}\right)^m+O_m\!\left(n^{m-1/5}\right).\]
\end{cor}

\section{Preliminaries}

For a graph $G$, let $V(G)$ and $E(G)$ denote its vertex set and edge set, respectively. For $x\in V(G)$, let $G-x$ denote the graph obtained from
$G$ by deleting $x$ and all edges incident with $x$. For $xy\in E(G)$, let $G-xy$ denote the graph obtained by deleting the edge $xy$.
A \emph{$k$-path} in $G$ is a sequence of distinct vertices $v_0,v_1,\ldots,v_k$ such that $v_{i-1}v_i\in E(G)$ for every $1\leq i\leq k$. Thus a $k$-path has $k$ edges and $k+1$ vertices; in particular, $P_k$ is a $(k-1)$-path. A $0$-path consists of a single vertex and is called a trivial path.

A \emph{weighted graph} is a pair $(G,w)$, where $G$ is a graph and $w:E(G)\to[0,\infty)$ is an edge-weight function. We write $W\defeq\sum_{e\in E(G)}w(e)$
for the total edge weight of $(G,w)$. Whenever $H$ is a subgraph of $G$, we give each edge of $H$ the same weight as in $G$ and continue to denote the resulting edge-weight function by $w$. For convenience, we set $w(xy)=0$ whenever $xy\notin E(G)$.

For a path $P=v_0v_1\cdots v_k$, write $V(P)\defeq\{v_0,v_1,\ldots,v_k\}$
and define its \emph{product weight} by \[w(P)\defeq\prod_{i=1}^{k}w(v_{i-1}v_i).\]
We use the convention that an empty product is $1$. Thus every trivial path has product weight $1$.

For a fixed weighted graph $(G,w)$, we now define three weighted path sums. For $x\in V(G)$ and $k\geq 0$, let $\mathcal{P}_k(G;x)$ denote the set of all $k$-paths in
$G$ starting at $x$, and define \[A_k(G;x)\defeq\sum_{P\in\mathcal{P}_k(G;x)}w(P).\]
In particular, $A_0(G;x)=1.$ For distinct vertices $x,y\in V(G)$ and $k\geq 1$, let $\mathcal{P}_k(G;x,y)$ denote the set of all $k$-paths in $G$ from $x$
to $y$, and define \[R_k(G;x,y)\defeq\sum_{P\in\mathcal{P}_k(G;x,y)}w(P).\]
In particular, $R_1(G;x,y)=w(xy).$ Finally, for distinct vertices $x,y\in V(G)$ and $k\geq 0$, define
\[D_k(G;x,y)\defeq\sum_{j=0}^{k}\mathop{\sum}\limits_{\substack{
    P\in\mathcal{P}_j(G;x)\\
    Q\in\mathcal{P}_{k-j}(G;y)\\
    V(P)\cap V(Q)=\varnothing}}w(P)w(Q).\]
Thus, $D_k(G;x,y)$ is the weighted sum over all ordered pairs $(P,Q)$ of vertex-disjoint paths starting at $x$ and $y$, respectively, whose lengths sum to $k$. In particular, $D_0(G;x,y)=1.$

\section{Proof of Theorem~\ref{thm:weighted}}

We first establish three inequalities for the weighted path sums defined above. The case $a=0$ of the next lemma follows from \cite[Lemma 2.5]{CohenAntonirShapira2024}.

\begin{lem}\label{lem:one-root}
Let $(G,w)$ be a weighted graph with total edge weight $W$. For every $a\geq 0$, every vertex $x\in V(G)$, and every integer $k\geq 1$,
$A_k(G;x)+aA_{k-1}(G;x)\leq \left(\frac{W+a}{k}\right)^k.$
\end{lem}

\begin{proof}
We use induction on $k$. For $k=1$, since $A_0(G;x)=1$,
\[A_1(G;x)+aA_0(G;x)=\sum_{xy\in E(G)}w(xy)+a\leq W+a.\]
Suppose that $k\geq 2$, and set $d\defeq\sum_{xy\in E(G)}w(xy)$ and $H\defeq G-x.$
The total edge weight of $H$ is $W-d$. Decomposing each path according to
its first edge and applying the induction hypothesis to $H$, we obtain
\begin{align*}
  A_k(G;x)+aA_{k-1}(G;x)
  &=\sum_{xy\in E(G)}\bigg(w(xy)\bigl(A_{k-1}(H;y)+aA_{k-2}(H;y)\bigr)\bigg)\\
  &\leq d\left(\frac{W-d+a}{k-1}\right)^{k-1}\leq \left(\frac{W+a}{k}\right)^k.
\end{align*}
The last inequality follows from the arithmetic--geometric mean
inequality applied to $d$ and $k-1$ copies of
$(W-d+a)/(k-1)$.
\end{proof}

\begin{lem}\label{lem:two-root}
Let $(G,w)$ be a weighted graph with total edge weight $W$. For every $a\geq 0$, every pair of distinct vertices $x,y\in V(G)$, and every integer $k\geq 1$,
\[D_k(G;x,y)+R_k(G;x,y)+aD_{k-1}(G;x,y)\leq\frac{(W+a)^k}{(k+1)^{k-1}}.\]
\end{lem}

\begin{proof}
We use induction on $k$. For $k=1$, the condition $V(P)\cap V(Q)=\varnothing$ excludes the edge $xy$ from $D_1(G;x,y)$, while $R_1(G;x,y)=w(xy)$.
Hence $D_1(G;x,y)+R_1(G;x,y)$ counts the weight of each edge incident with $x$ or $y$ exactly once. Since $D_0(G;x,y)=1$,
\[D_1(G;x,y)+R_1(G;x,y)+aD_0(G;x,y)\leq W+a.\]
This proves the result for $k=1$.

Now let $k\geq 2$, and assume that the result holds for $k-1$. Set $d\defeq\sum_{xz\in E(G)}w(xz) $ and $ H \defeq  G-x.$
The total edge weight of $H$ is $W-d$. For $D_k(G;x,y)$, first consider the case in which the path starting at $x$ is trivial. The other path is then a $k$-path in $H$ starting at $y$, and these terms sum to $A_k(H;y)$. If the path starting at $x$ is nontrivial, let $xz$ be its first edge. Since the two paths are vertex-disjoint, $z\neq y$. Removing $x$ and the edge $xz$ gives a pair of paths counted by $D_{k-1}(H;z,y)$.
Thus, \[D_k(G;x,y)=A_k(H;y)+\sum_{\substack{xz\in E(G)\\ z\neq y}}\bigg(w(xz)D_{k-1}(H;z,y)\bigg).\]
Similarly, by considering the first edge of a path from $x$ to $y$, we have
\[R_k(G;x,y)=\sum_{\substack{xz\in E(G)\\ z\neq y}}\bigg(w(xz)R_{k-1}(H;z,y)\bigg).\]
The same decomposition applied to $D_{k-1}(G;x,y)$ gives
\[D_{k-1}(G;x,y)=A_{k-1}(H;y)+\sum_{\substack{xz\in E(G)\\ z\neq y}}\bigg(w(xz)D_{k-2}(H;z,y)\bigg).\]

Combining these three identities, we obtain
\begin{align*}
  &D_k(G;x,y)+R_k(G;x,y)+aD_{k-1}(G;x,y)\\
  &\quad=A_k(H;y)+aA_{k-1}(H;y)+\sum_{\substack{xz\in E(G)\\ z\neq y}}\bigg(w(xz)\bigl(D_{k-1}(H;z,y)+R_{k-1}(H;z,y)+aD_{k-2}(H;z,y)\bigr)\bigg).
\end{align*}

Since $\sum_{\substack{xz\in E(G)\\ z\neq y}}w(xz)=d-w(xy)\leq d,$ Lemma~\ref{lem:one-root} and the induction hypothesis give
\begin{align*}
  D_k(G;x,y)+R_k(G;x,y)+aD_{k-1}(G;x,y)\notag\leq \left(\frac{W-d+a}{k}\right)^k+d\,\frac{(W-d+a)^{k-1}}{k^{k-2}}.
\end{align*}

We complete the proof by showing that
\[\left(\frac{W-d+a}{k}\right)^k+d\,\frac{(W-d+a)^{k-1}}{k^{k-2}}\leq\frac{(W+a)^k}{(k+1)^{k-1}}.\]
If $W+a=0$, then $W=a=d=0$, and the inequality is immediate. Assume that$W+a>0$, and set $t\defeq\frac{d}{W+a}.$ Since $0\leq d\leq W\leq W+a$, we have $0\leq t\leq 1$ and
\[W-d+a=(1-t)(W+a).\]
Dividing the left-hand side of the inequality above by $(W+a)^k$, we obtain \[f_k(t)\defeq\frac{(1-t)^{k-1}}{k^k}\bigl(1-t+k^2t\bigr).\]
A direct calculation gives \[f_k'(t)=\frac{k-1}{k^{k-1}}(1-t)^{k-2}\bigl(1-(k+1)t\bigr).\]
Thus, $f_k$ is increasing on $\left[0,1/(k+1)\right]$ and decreasing on $\left[1/(k+1),1\right]$.
Hence \[f_k(t)\leq f_k\left(\frac{1}{k+1}\right)=\frac{1}{(k+1)^{k-1}}.\]
This completes the proof.
\end{proof}

\begin{lem}\label{lem:rooted}
Let $(G,w)$ be a weighted graph with total edge weight $W$. Let $a\geq 0$ and suppose that $w(e)\leq a$ for every $e\in E(G).$
Then, for every pair of distinct vertices $x,y\in V(G)$ and every integer $k\geq 1$,
\[D_k(G;x,y)+(k+1)R_k(G;x,y)\leq\frac{(W+a)^k}{(k+1)^{k-1}}.\]
\end{lem}

\begin{proof}
We first show that $kR_k(G;x,y)\leq aD_{k-1}(G;x,y).$ Let $P=v_0v_1\cdots v_k\in\mathcal{P}_k(G;x,y), $where $v_0=x$ and $v_k=y$. For each $i\in\{1,\ldots,k\}$, remove the edge $v_{i-1}v_i$ and set $P_1=v_0v_1\cdots v_{i-1}$ and $P_2=v_kv_{k-1}\cdots v_i.$
The paths $P_1$ and $P_2$ start at $x$ and $y$, respectively. They are vertex-disjoint, and their lengths sum to $k-1$. Thus, $(P_1,P_2)$ is counted by $D_{k-1}(G;x,y)$.

No ordered pair $(P_1,P_2)$ is obtained more than once. Indeed, the length of $P_1$ determines $i$, and $P$ is recovered by following $P_1$, then the deleted edge, and then $P_2$ in reverse. Also, \[w(P)=w(v_{i-1}v_i)w(P_1)w(P_2)\leq a~w(P_1)w(P_2).\]
Since every such pair $(P_1,P_2)$ is counted by $D_{k-1}(G;x,y)$, summing over all $P$ and $i$ gives
\[kR_k(G;x,y)=\sum_{P\in\mathcal{P}_k(G;x,y)}\sum_{i=1}^{k}w(P)\leq aD_{k-1}(G;x,y).\]

It now follows from Lemma~\ref{lem:two-root} that
\begin{align*}
  D_k(G;x,y)+(k+1)R_k(G;x,y)
  &=D_k(G;x,y)+R_k(G;x,y)+kR_k(G;x,y)\\
  &\leq D_k(G;x,y)+R_k(G;x,y)+aD_{k-1}(G;x,y)\\
  &\leq \frac{(W+a)^k}{(k+1)^{k-1}}.
\end{align*}
\end{proof}

We need the following optimality identity of Heath, Martin, and Wells
\cite[Lemma 4.3]{HeathMartinWells2025}.

\begin{lem}[Heath, Martin, and Wells~\cite{HeathMartinWells2025}]
\label{lem:HMW-optimality}
Let $m$ and $s$ be positive integers, and let $H_1,\ldots,H_s$ be graphs with $m$ edges. Let $K$ be a finite complete graph with at least two vertices, and let $\gamma_1,\ldots,\gamma_s$ be real numbers. Set $\mathcal O\defeq\max_{\nu\in\Delta^K}\sum_{i=1}^{s}\gamma_i\cdot \beta(\nu;H_i).$
If $\mu\in\Delta^K$ attains this maximum, then, for every edge $e\in E(K)$,
\[m\cdot \mathcal O\cdot \mu(e)=\sum_{i=1}^{s}\gamma_i \sum_{\substack{H'\in\cC(K,H_i)\\e\in E(H')}}\mu(H').\]
\end{lem}

We also use the following weighted cycle inequality of Lv, Gy\H{o}ri, He, Salia, Tompkins, and Zhu \cite[Theorem 2]{LvEtAl2024}.

\begin{lem}[Lv, Gy\H{o}ri, He, Salia, Tompkins, and Zhu \cite{LvEtAl2024}]
\label{lem:cycle}
For every finite complete graph $K$, every integer $m\geq 3$, and every $\mu\in\Delta^K$, $\beta(\mu;C_m)\leq\frac{1}{m^m}.$
Equality holds if and only if $\mu$ is the uniform measure on the edge set of a copy of $C_m$ in $K$.
\end{lem}

\begin{proof}[\textbf{Proof of Theorem~\ref{thm:weighted}}]
Fix a finite complete graph $K$. If $|V(K)|<m$, then $K$ contains neither a copy of $C_m$ nor a copy of $P_{m+1}$. Thus \eqref{eq:strong} is
immediate, and equality in the inequality of Conjecture~\ref{conj:HMW} cannot hold. We may therefore assume that $|V(K)|\geq m$.

Set \[\mathcal O^{'} \defeq \max_{\nu\in\Delta^K} \bigl(2m\cdot \beta(\nu;C_m)+2\beta (\nu;P_{m+1})\bigr).\]
The maximum exists because $\Delta^K$ is compact and the function being maximized is continuous. Let $\mu\in\Delta^K$ attain $\mathcal O^ {'}$.
Let $xy\in E(K)$ satisfy $\mu(xy)=\max_{e\in E(K)}\mu(e),$ and set $a\defeq\mu(xy).$
Since $\sum_{e\in E(K)}\mu(e)=1$, we have $a>0$.
Let $G\defeq K-xy$ and give each edge $e\in E(G)$ weight $w(e)\defeq\mu(e).$
Then $(G,w)$ has total edge weight $W=1-a,$ and $w(e)\leq a$ for every $e\in E(G).$

Deleting $xy$ from a copy of $C_m$ containing $xy$ gives an $(m-1)$-path in $G$ from $x$ to $y$. Conversely, adding $xy$ to such a path gives a unique copy of $C_m$. Hence
\[\sum_{\substack{C\in\cC(K,C_m)\\xy\in E(C)}}\mu(C)=aR_{m-1}(G;x,y).\]

Similarly, deleting $xy$ from a copy of $P_{m+1}$ containing $xy$ gives two vertex-disjoint paths in $G$, starting at $x$ and $y$, whose lengths
sum to $m-1$. Conversely, adding $xy$ to any pair counted by $D_{m-1}(G;x,y)$ gives a unique copy of $P_{m+1}$.
The roots $x$ and $y$ fix the order of the two paths, so there is no extra factor of $2$.
Therefore,
\[\sum_{\substack{P\in\cC(K,P_{m+1})\\xy\in E(P)}}\mu(P)=aD_{m-1}(G;x,y).\]

Apply Lemma~\ref{lem:HMW-optimality} with $H_1=C_m,H_2=P_{m+1},\gamma_1=2m,\gamma_2=2.$
For the edge $xy$, the lemma gives
\begin{align*}
  m\cdot \mathcal O^{'}\cdot a
  &=2m\sum_{\substack{C\in\cC(K,C_m)\\xy\in E(C)}}\mu(C)+2\sum_{\substack{P\in\cC(K,P_{m+1})\\xy\in E(P)}}\mu(P)\\
  &=2maR_{m-1}(G;x,y)+2aD_{m-1}(G;x,y).
\end{align*}
Since $a>0$, dividing by $a$ gives
\[\mathcal O^{'}=\frac{2}{m}\bigl(D_{m-1}(G;x,y)+mR_{m-1}(G;x,y)\bigr).\]

Since $W+a=1$, Lemma~\ref{lem:rooted}, applied with $k=m-1$, gives
\[D_{m-1}(G;x,y)+mR_{m-1}(G;x,y)\leq\frac{1}{m^{m-2}}.\]
It follows that \[\mathcal O^{'}\leq\frac{2}{m}\cdot\frac{1}{m^{m-2}}=\frac{2}{m^{m-1}}.\]
Since $\mathcal O^{'}$ is the maximum over all measures in $\Delta^K$, this proves \eqref{eq:strong}.

It remains to prove the equality statement in Conjecture~\ref{conj:HMW}. Suppose that $\lambda\in\Delta^K$ satisfies
\[2m\cdot \beta(\lambda;C_m)+\beta(\lambda;P_{m+1})=\frac{2}{m^{m-1}}.\]
By \eqref{eq:strong},
\[\frac{2}{m^{m-1}}=2m\cdot \beta(\lambda;C_m)+\beta(\lambda;P_{m+1})\leq 2m\cdot \beta(\lambda;C_m)+2\cdot \beta(\lambda;P_{m+1})\leq \frac{2}{m^{m-1}}.\]
It follows that $\beta(\lambda;P_{m+1})=0$ and $\beta(\lambda;C_m)=\frac{1}{m^m}.$
By Lemma~\ref{lem:cycle}, $\lambda$ is the uniform measure on the edge set of a copy of $C_m$ in $K$.
Conversely, if $\lambda$ is uniform on the edge set of a copy of $C_m$, then $\beta(\lambda;C_m)=1/m^m$ and $\beta(\lambda;P_{m+1})=0$, so equality holds.
\end{proof}

\end{document}